\documentclass[12pt,a4paper]{article}
\usepackage[utf8]{inputenc}
\usepackage[english]{babel}
\usepackage{amsmath}
\usepackage{amsfonts}
\usepackage{amssymb}
\usepackage{amsthm}
\usepackage{graphicx}
\usepackage[left=2cm,right=2cm,top=2cm,bottom=2cm]{geometry}
\author{Carlos Fresneda-Portillo}
\theoremstyle{plain}
\newtheorem{theorem}{Theorem}[section]
\newtheorem{lemma}[theorem]{Lemma}
\newtheorem{corollary}[theorem]{Corollary}

\newtheorem{cond}[theorem]{Condition}
\theoremstyle{definition}

\theoremstyle{remark}
\newtheorem{remark}{Remark}[section]

\newcommand{\Div}[1]{\text{div}\,}
\newcommand\Label[1]{&\refstepcounter{equation}(\theequation)\ltx@label{#1}&}
\numberwithin{equation}{section}
\begin{document}

\title{Boundary-Domain Integral Equations for the Mixed Problem for the Diffusion Equation in Inhomogeneous Media based on a New Family of Parametrices on Unbounded Domains}

\author{C. Fresneda-Portillo}

\maketitle

\abstract*{A system of Boundary-Domain Integral Equations is derived from the mixed (Dirichlet-Neumann) boundary value problem for the diffusion equation in inhomogeneous media defined on an unbounded domain. This paper extends the work introduced in \cite{carloscomp} to unbounded domains. 
Mapping properties of parametrix-based potentials on weighted Sobolev spaces are analysed. Equivalence between the original boundary value problem and the system of BDIEs is shown. Uniqueness of solution of the BDIEs is proved using Fredholm Alternative and compactness arguments adapted to weigthed Sobolev spaces.}

\section{Introduction}
Boundary Domain Integral Equations appear naturally when applying the Boundary Integral Method to boundary value problems with variable coefficient. These class of boundary value problems has a wide range of applications in Physics or Engineering, such as, heat transfer in non-homogeneous media \cite{ravnik}, motion of laminar fluids with variable viscosity \cite{green2}, or even in the acoustic scattering by inhomogeneous anisotropic obstacle \cite{acoustics}. 

The popularity of the Boundary Integral Method is due to the reduction of the discretisation domain. For example, if the boundary value problem (BVP) is defined on a three dimensional domain, then, the boundary integral method reduces the BVP to an equivalent system of boundary integral equations (BIEs) defined only on the \textit{boundary} of the domain. However, this requires an explicit fundamental solution of the partial differential equation appearing in the BVP. Although these fundamental solutions may exist, they might not always be available explicitly for PDEs with variable coefficients. To overcome this obstacle, one can construct a \textit{parametrix} using the known fundamental solution. A discussion on fundamental solution existence theorems, algorithms for constructing fundamental solutions and parametrices is available in \cite{pomp}; for classical examples of derivation of  Boundary Domain Integral Equations refer to \cite{mikhailov1} for the diffusion equation with variable coefficient in bounded domains in $\mathbb{R}^{3}$; \cite{carloscomp} for the same problem applying a different parametrix; \cite{carloszenebe} for the Dirichlet problem in $\mathbb{R}^{2}$ and \cite{carlosstokes} for the mixed problem for the compressible Stokes system, as an example of derivation of BDIEs from a PDE system. 

The introduction of a parametrix for BVPs with variable coefficient lead to a system of integral equations not only defined on the boundary but also in the domain.  Still, one can transform domain integrals into boundary integrals applying the methods shown in \cite{RIM}. These methods help to preserve the reduction of dimension while also remove singularities appearing in the domain integrals. 

The approximation of numerical solutions of BDIEs is a relevant problem nowadays. In particular, the very recent article \cite{numerics2d} focuses on the solution of the analogous mixed BVP presented in this paper in $\mathbb{R}^{2}$. In \cite{2dnumerics}, the authors show that it is possible to obtain linear convergence with respect to the number of quadrature curves, and in some cases, exponential convergence. Analogous research in 3D shows the successful implementation of fast algorithms to obtain the solution of boundary domain integral equations, see \cite{ravnik, numerics, sladek}. Furthermore, the authors \cite{chapko} show the application of the Boundary Domain Integral Equation Method to the study of inverse problems with variable coefficients. 

A parametrix is not unique, see discussion on \cite[Section 1]{carloscomp}. The study of different  parametrices is adventageous to construct parametrices for PDE systems. Moreover, numerical methods may work with one parametrix more efficiently than with another. However, before attempting numerical experiments, results on the existence and uniqueness of solution need to be established what is the purpose of this paper.  

\textit{In this paper,} we extend the results presented in \cite{carloscomp} to unbounded domains which employ a different parametrix from the one used in \cite{exterior}. 

In unbounded domains, the mixed problem is set in weighted Sobolev spaces to allow constant functions in unbounded domains to be possible solutions of the problem. Hence, all the mapping properties of the parametrix based potential operators are shown in weighted Sobolev spaces. 

An analysis of the uniqueness of the BDIES is performed by studying the Fredholm properties of the matrix operator which defines the system. Unlike for the case of bounded domains, the Rellich compactness embeding theorem is not available for Sobolev spaces defined over unbounded domains. Nevertheless, we present a lemma to reduce the remainder operator to two operators: one invertible and one compact. Therefore, we can still benefit from the Fredholm Alternative theory to prove uniqueness of the solution.

%

\section{Weighted Sobolev spaces}
Let $\Omega=\Omega^{+}$ be an unbounded exterior connected domain. Let $\Omega^{-}:=\mathbb{R}^{3}\smallsetminus \overline{\Omega}^{+}$ the complementary (bounded) subset of $\Omega$. The boundary $S := \partial\Omega$ is simply connected, closed and infinitely differentiable, $ S\in \mathcal{C}^{\infty}$. Furthermore, $S :=\overline{S}_{N}\cup \overline{S}_{D}$ where both $S_{N}$ and $S_{D}$ are non-empty, connected disjoint submanifolds of $S$. The border of these two submanifolds is also infinitely differentiable: $\partial S_{N}= \partial S_{D}\in\mathcal{C}^{\infty}$. 

With regards to function spaces that we employ on this paper, $\mathcal{D}(\Omega):=C^{\infty}_{comp}(\Omega)$ denotes the space of test functions, and $\mathcal{D}^{*}(\Omega)$ denotes the space of distributions or generalised functions. We also use Sobolev spaces $H^{s}(\Omega)$, Bessel potential spaces on the boundary of the domain 
$H^{s}(\partial\Omega)$, where $s\in\mathbb{R}$ (see e.g. \cite{mclean, hsiao} for more details). We recall that $H^{s}$ coincide with the Sobolev-Slobodetski spaces $W^{2,s}$ for any non-negative $s$. We denote by $\widetilde{H}^{s}(\Omega)$ 
the subspace of ${H}^{s}(\mathbb{R}^{3})$, $\widetilde{H}^{s}(\Omega):=\{g:g\in
H^{s}(\mathbb{R}^{3}),~\textrm{supp}~g\subset\overline{\Omega}\}$. Note that the space $H^{s}(\Omega)$ is characterised as all distributions from $H^{s}(\mathbb{R}^{3})$ restricted to $\Omega$, $H^{s}(\Omega)=\{r_{_{\Omega}}g:g\in
H^{s}(\mathbb{R}^{3})\}$ where $r_{_{\Omega}}$ denotes the restriction operator on $\Omega.$ 

To ensure uniquely solvability of the BVPs in exterior domains, we will use \textit{weighted Sobolev spaces} with weight $\omega(x) = (1+\vert x\vert^{2})^{1/2}$, (see e.g., \cite{exterior}). Let
\begin{center}
$L^{2}(\omega^{-1}; \Omega) = \lbrace g: \omega^{-1}g\in L^{2}(\Omega)\rbrace,$
\end{center}
be the weighted Lebesgue space and $\mathcal{H}^{1}(\Omega)$ the following weighted Sobolev (Beppo-Levi) space constructed using the $L^{2}(\omega^{-1}; \Omega)$ space
\begin{center}
$\mathcal{H}^{1}(\Omega):=\lbrace g\in L^{2}(\omega^{-1}; \Omega): \nabla g\in L^{2}(\Omega)\rbrace$ 
\end{center} endowed with the corresponding norm
\begin{center}
$\parallel g \parallel^{2}_{\mathcal{H}^{1}(\Omega)} := \parallel \omega ^{-1}g\parallel^{2}_{L^{2}(\Omega)} + \parallel \nabla g \parallel ^{2}_{L^{2}(\Omega)}.$
\end{center}

Taking into account that $\mathcal{D}(\overline{\Omega})$ is dense in $H^{1}(\Omega)$ it is easy to prove that $\mathcal{D}(\overline{\Omega})$ is dense in ${\mathcal{H}}^{1}(\Omega)$. For further details, cf. \cite[p.3]{exterior} and more references therein. 

If $\Omega$ is unbounded, then the seminorm
\begin{center}
$\vert {g}\vert_{{\mathcal{H}}^{1}(\Omega)}:= \parallel \nabla {g}\parallel_{{L}^{2}(\Omega)},$
\end{center}
is equivalent to the norm $\parallel {g} \parallel_{{\mathcal{H}}^{1}(\Omega)}$ in ${\mathcal{H}}^{1}(\Omega)$ \cite[ Chapter XI, Part B, \S 1]{lions}. On the contrary, if $\Omega^{-}$ is bounded, then ${\mathcal{H}}^{1}(\Omega^{-})={H}^{1}(\Omega^{-})$. If $\Omega'$ is a bounded subdomain of an unbounded domain $\Omega$ and ${g}\in {\mathcal{H}}^{1}(\Omega)$, then ${g}\in {H}^{1}(\Omega')$. 

Let us introduce $\widetilde{{\mathcal{H}}}^{1}(\Omega)$ as the completion of ${\mathcal{D}}(\Omega)$ in ${\mathcal{H}}^{1}(\mathbb{R}^{3})$; let $\widetilde{{\mathcal{H}}}^{-1}(\Omega) := [{\mathcal{H}}^{1}(\Omega)]^{*}$ and ${\mathcal{H}}^{-1}(\Omega) := [ \widetilde{{\mathcal{H}}}^{1}(\Omega)]^{*}$ be the corresponding dual spaces. Evidently, the space ${L}^{2}(\omega; \Omega)\subset {\mathcal{H}}^{-1}(\Omega)$. 

For any generalised function ${g}$ in ${\widetilde{{\mathcal{H}}}}^{-1}(\Omega)$, we have the following representation property, see \cite[Section 2]{exterior}, $g_{j}=\partial_{i} g_{ij} + g_{j}^{0}, \quad g_{ij}\in L^{2}(\mathbb{R}^{3})$ and are zero outside the domain $\Omega$, whereas $g_{j}^{0}\in L^{2}(\omega;\Omega)$. Consequently, ${\mathcal{D}}(\Omega)$ is dense in ${\widetilde{{\mathcal{H}}}}^{-1}(\Omega) $ and ${\mathcal{D}}(\mathbb{R}^{3})$ is dense in ${{\mathcal{H}}}^{-1}(\mathbb{R}^{3})$.

\section{Traces, conormal derivatives and Green identities}

We consider the following differential operator
\begin{equation}\label{ch5operatorA}
\mathcal{A}u(x):=\sum_{i=1}^{3}\dfrac{\partial}{\partial x_{i}}\left(a(x)\dfrac{\partial u(x)}{\partial x_{i}}\right)\in \Omega,
\end{equation}
 where $a(x)\in \mathcal{C}^{2}$, $a(x)>0$, is a variable coefficient. It is easy to see that if $a\equiv 1$ then, the operator $\mathcal{A}$ becomes the Laplace operator $\Delta$.
 
Here and thereafter, we will assume the following condition on the coefficient $a(x)$.
\begin{cond}\label{cond0}
The coefficient $a(x)$ belongs to the space $L^{\infty}(\Omega)$. Furthermore, there exist two positive constants, $C_{1}$ and $C_{2}$, such that:
\begin{equation}\label{ch5conda(x)1}
0 < C_{1} < a(x) < C_{2}.
\end{equation}
\end{cond}
 
The Condition \ref{cond0} is necessary so that the operator $\mathcal{A}$ acting on $u\in \mathcal{H}^{1}(\Omega)$ is well defined in the weak sense. Hence, we define the operator $\mathcal{A}$ in the weak sense as
\begin{equation}\label{ch5exbilinearA}
\langle \mathcal{A}u, v \rangle := - \langle a\nabla u, \nabla v\rangle =-\mathcal{E}(u,v) \quad \forall v \in \mathcal{D}(\Omega),
\end{equation}
where
\begin{equation}\label{ch5exE}
\mathcal{E}(u,v) : = \int_{\Omega} E(u,v) (x) dx, \quad\quad E(u,v)(x):= a(x) \nabla u(x)\nabla v(x). 
\end{equation}

Note that the functional $\mathcal{E}(u,v):\mathcal{H}^{1}(\Omega)\times \widetilde{\mathcal{H}}^{1}(\Omega)\longrightarrow \mathbb{R}$ is continuous under Condition \ref{cond0}. Therefore, the density  of $\mathcal{D}(\Omega)$ in $\widetilde{\mathcal{H}}^{1}(\Omega)$ implies the continuity of the operator  $\mathcal{A}:\mathcal{H}^{1}(\Omega)\longrightarrow\mathcal{H}^{-1}(\Omega)$ in \eqref{ch5exbilinearA} which gives the weak form of the operator $\mathcal{A}$.
 
For a scalar function $w\in H^{1}(\Omega)$ in virtue of the trace theorem it follows that $\gamma^{\pm}w\in H^{1/2}(S)$ where the trace operators from $\Omega^{\pm}$ to $S$ are denoted by $\gamma^{\pm}$ respectively. Consequently, if $w\in H^{1}(\Omega)$, then $w\in \mathcal{H}^{1}(\Omega)$ and it follows that $\gamma^{\pm}w\in H^{1/2}(S)$, (see, e.g., \cite{mclean, traces}). 
For $u\in H^{s}(\Omega); s>3/2$, we can define by $T^{\pm}$ the  conormal derivative operator acting on $S$ understood in the classical sense:
\begin{equation}\label{ch5conormal}
T^{\pm}[u(x)] :=\sum_{i=1}^{3}a(x)n_{i}(x)\gamma^{\pm}\left( \dfrac{\partial u}{\partial x_{i}}\right)=a(x)\gamma^{\pm}\left( \dfrac{\partial u(x)}{\partial n(x)}\right),
\end{equation}
\noindent where $n(x)$ is the exterior unit normal vector to the domain $\Omega$ at a point $x\in S$.

However, for $u\in \mathcal{H}^{1}(\Omega)$ (as well as for $u\in H^{1}(\Omega)$), the classical co-normal derivative operator may not exist on the trace sense. This issue is overcome by introducing the following function space for the operator $\mathcal{A}$, (cf. \cite{exterior})
\begin{equation}\label{ch5spaceHA}
\mathcal{H}^{1,0}(\Omega;\mathcal{A}) := \lbrace g\in \mathcal{H}^{1}(\Omega): \mathcal{A}g\in L^{2}(\omega; \Omega)\rbrace
\end{equation}
endowed with the norm
\begin{center}
$\parallel g\parallel ^{2}_{\mathcal{H}^{1,0}(\Omega;\mathcal{A})} := \parallel g \parallel ^{2}_{\mathcal{H}^{1}(\Omega)} + \parallel \omega \mathcal{A}g\parallel^{2}_{L^{2}(\Omega)}.$
\end{center}

Now, if a distribution $u\in \mathcal{H}^{1,0}(\Omega; \mathcal{A})$, we can appropriately define the conormal derivative $T^{+}u\in H^{-1/2}(S)$ using the Green's formula, cf. \cite{mclean, exterior},
\begin{equation}\label{ch5green1}
\langle T^{+}u, w\rangle_{S}:= \pm \int_{\Omega^{\pm}}[(\gamma^{+}_{-1}\omega)\mathcal{A}u +E(u,\gamma_{-1}^{+}w)]\,\,dx,\,\, \text{for all}\,\, w\in H^{1/2}(S),
\end{equation}
where $\gamma_{-1}^{+}: H^{1/2}(S)\rightarrow \mathcal{H}^{1}(\Omega)$ is a continuous right inverse to the trace operator $\gamma^{+}:\mathcal{H}^{1}(\Omega) \longrightarrow H^{1/2}(S)$ while the brackets $\langle u,v \rangle_{S}$ represent the duality brackets of the spaces $H^{1/2}(S)$ and $H^{-1/2}(S)$ which coincide with the scalar product in $L^{2}(S)$ when $u,v\in L^{2}(S)$.

The operator $T^{+}:\mathcal{H}^{1,0}(\Omega; \mathcal{A})\longrightarrow H^{-1/2}(S)$ is bounded and gives a continuous extension on $\mathcal{H}^{1,0}(\Omega; \mathcal{A})$ of the classical co-normal derivative operator \eqref{ch5conormal}. We remark that when $a\equiv 1$, the operator $T^{+}$ becomes the continuous extension on $\mathcal{H}^{1,0}(\Omega; \Delta)$ of the classical normal derivative operator $T^{+}_{\Delta}u = \partial_{n} u:= n\cdot \nabla u$. 

In a similar manner as in the proof \cite[Lemma 4.3]{mclean} or \cite[Lemma 3.2]{costabel}, the first Green identity holds for a distribution $u\in\mathcal{H}^{1,0}(\Omega;\mathcal{A})$
\begin{equation}\label{ch5GF1}
\langle T^{+}u ,\gamma^{+}v\rangle_{S}=\displaystyle\int_{\Omega}[v\mathcal{A}u + E(u,v)] dx, \quad \forall v\in\mathcal{H}^{1}(\Omega).
\end{equation}

Applying the identity \eqref{ch5GF1} to $u,v\in\mathcal{H}^{1,0}(\Omega;\mathcal{A})$, exchanging roles of $u$ and $v$, and then subtracting the one from the other, we arrive to the following second Green identity, see e.g. \cite{mclean}

\begin{equation}\label{ch5secondgreen}
\displaystyle\int_{\Omega}\left[v\mathcal{A}u - u\mathcal{A}v\right]\,dx= \int_{S}\left[\gamma^{+}v\,T^{+}u-\gamma^{+}u\,T^{+}v\right]\,dS(x).
\end{equation}

\section{Boundary Value Problem}
Now that we have shown that if $u \in  \mathcal{H}^{1,0}(\Omega;\mathcal{A})$, then its trace and its conormal derivative are well defined, it is possible to formulate the mixed problem for the operator $\mathcal{A}$ for which we aim to derive an equivalent of system of boundary-domain integral equations (BDIEs).  

\textbf{Mixed problem} Find $u\in \mathcal{H}^{1,0}(\Omega;\mathcal{A})$ such that
\begin{align}
\label{ch5BVP1}\mathcal{A}u&=f,\quad\text{in}\hspace{0.1cm}\Omega;\\
\label{ch5BVPD}r_{S_{D}}\gamma^{+}u&=\phi_{0},\quad\text{on}\hspace{0.1cm} S_{D};\\
\label{ch5BVPN}r_{S_{N}}T^{+}u&=\psi_{0},\quad\text{on} \hspace{0.1cm}S_{N}.
\end{align}
where $f\in L^{2}(\omega,\Omega)$, $\phi_{0}\in H^{1/2}(S_{D})$ and $\psi_{0}\in H^{-1/2}(S_{N})$. 
         
The previous BVP can be represented with the following operator equation ${\mathcal{A}_{M}u =\mathcal{F}_{M},}$ where 
\begin{align*}
\mathcal{A}_{M}:\mathcal{H}^{1,0}(\Omega;\mathcal{A})&\longrightarrow L^{2}(\omega,\Omega)\times H^{1/2}(S_{D})\times H^{-1/2}(S_{N});\\
u &\longrightarrow \mathcal{A}_{M}u : = (\mathcal{A}u, \gamma^{+}u, T^{+}u),
\end{align*}
and $\mathcal{F}_{M}:=(f,\phi_{0},\psi_{0})\in L^{2}(\omega,\Omega)\times H^{1/2}(S_{D})\times H^{-1/2}(S_{N})$. 
The following result is well known and it has been proven \cite[Appendix A]{exterior} by using variational settings and the Lax Milgram lemma. 

\begin{theorem}\label{ch5thinv0}If $a(x)\in L^{\infty}(\Omega)$ and $a(x)>0$, then the mixed problem \eqref{ch5BVP1}-\eqref{ch5BVPN} is uniquely solvable in $\mathcal{H}^{1,0}(\Omega; \mathcal{A})$ and the inverse operator of $\mathcal{A}_{M}$ is continuous 
\begin{align*}
\mathcal{A}_{M}^{-1}&:L^{2}(\omega,\Omega)\times H^{1/2}(S_{D})\times H^{-1/2}(S_{N})\longrightarrow \mathcal{H}^{1,0}(\Omega;\mathcal{A}).\\
\end{align*}
\end{theorem}
It is clear that hypotheses of the Theorem \ref{ch5thinv0} are satisfied under the assumption of Condition \ref{cond0}. Hence, the mixed BVP problem \eqref{ch5BVP1}-\eqref{ch5BVPN} is uniquely solvable. 

\section{Parametrices and remainders}
 We define a parametrix (Levi function) $P(x,y)$ for the differential operator $\mathcal{A}$ differentiating with respect to $x$, as a function on two variables that satisfies
 \begin{equation}\label{parametrixdef}
 \mathcal{A}P(x,y) = \delta(x-y)+R(x,y).
 \end{equation}
 where $\delta(.)$ is the Dirac distribution and the term $R(x,y)$ is a weakly singular distribution, i.e. $\mathcal{O}(\vert x-y\vert^{-2})$, so-called remainder. 
A given operator $\mathcal{A}$ may have more than one parametrix. For example, the parametrix 
\begin{equation*}\label{ch4P2002}
P^y(x,y)=\dfrac{1}{a(y)} P_\Delta(x-y),\hspace{1em}x,y \in \mathbb{R}^{3},
\end{equation*} 
was employed in \cite{localised, mikhailov1}, for the operator $\mathcal{A}$, given in \eqref{ch5operatorA},  where
\begin{equation*}\label{ch4fundsol}
P_{\Delta}(x-y) = \dfrac{-1}{4\pi \vert x-y\vert}
\end{equation*}
is the fundamental solution of the Laplace operator.
The remainder corresponding to the parametrix $P^{y}$ is 
\begin{equation*}
\label{ch43.4} R^y(x,y)
=\sum\limits_{i=1}^{3}\frac{1}{a(y)}\, \frac{\partial a(x)}{\partial x_i} \frac{\partial }{\partial x_i}P_\Delta(x-y)
\,,\;\;\;x,y\in {\mathbb R}^3.
\end{equation*}
\textit{In this paper}, we consider the parametrix $P^{x}$ used in \cite{carloscomp, carlosL}, where analogous results to the ones presented in the upcoming sections have been obtained in \textit{bounded} domains with smooth and Lipschitz boundary.

The parametrix $P^{x}$ is defined as follows:
\begin{align}\label{ch4Px}
P(x,y):=P^x(x,y)=\dfrac{1}{a(x)} P_\Delta(x-y),\hspace{1em}x,y \in \mathbb{R}^{3},
\end{align}
which leads to the corresponding remainder 
\begin{align*}\label{ch4remainder}
R(x,y) =R^x(x,y) &= 
-\sum\limits_{i=1}^{3}\dfrac{\partial}{\partial x_{i}}
\left(\frac{1}{a(x)}\dfrac{\partial a(x)}{\partial x_{i}}P_{\Delta}(x,y)\right)\\
& =
-\sum\limits_{i=1}^{3}\dfrac{\partial}{\partial x_{i}}
\left(\dfrac{\partial (\ln a(x))}{\partial x_{i}}P_{\Delta}(x,y)\right),\hspace{0.5em}x,y \in \mathbb{R}^{3}.
\end{align*}
Due to the smoothness of the variable coefficient $a(x)$, both remainders $R_x$ and $R_y$ are weakly singular, i.e., $
R^x(x,y),\,R^y(x,y)\in \mathcal{O}(\vert x-y\vert^{-2}).$

Let us remark that this parametrix $P^{x}(x,y)$ is different from the parametrix 
\begin{equation*}\label{ch5P2002}
P^y(x,y)=\dfrac{1}{a(y)} P_\Delta(x-y),\hspace{1em}x,y \in \mathbb{R}^{3},
\end{equation*} 
which has been used to derive analogous results to those in this paper, in \cite{exterior}. 

The parametrix $P^{y}$ has been widely analysed in the literature, see \cite{localised, mikhailovlipschitz, numerics, mikhailov1, miksolandreg}. The difference between both parametrices relies on the dependence from the variable of the coefficient $a(x)$ or $a(y)$. Clearly, choosing a parametrix involving $a(y)$ simplifies the expression of the remainder as the coefficient $a(y)$ acts as a constant when differentiating with respect to $x$ which is the variable of differentiation of the operator $\mathcal{A}$. However, for some PDE problems, it is not always possible to obtain a parametrix that depends exclusively on $a(y)$ and not on $a(x)$. This is the case of the Stokes system, see \cite{carlosstokes}. Hence, the usefulness of the analysis of the family of parametrices depending on $a(x)$. 

\section{Volume and surface potentials}
Boundary-domain integral equations are usually formulated in terms of parametrix-based surface and volume potential operators. In this section, the surface and volume potentials based on the parametrix $P^{x}$ are introduced. We analyse their mapping properties in weighted Sobolev spaces. Additional boundedness conditions are often imposed on the variable coefficient $a(x)$ in order to prove the boundedness properties of the potential operators. 

\begin{cond}\label{cond1} We will assume the following condition further on unless stated otherwise:
\begin{equation}\label{ch5conda(x)2}
a\in \mathcal{C}^{1}(\mathbb{R}^{3})\quad and \quad \omega\nabla a \in L^{\infty}(\mathbb{R}^{3}).
\end{equation}
\end{cond}

\begin{remark}\label{ch5remmult}
If the coefficient $a(x)$ satisfies \eqref{ch5conda(x)1} and \eqref{ch5conda(x)2}, then $$\parallel ga\parallel_{\mathcal{H}^{1}(\Omega)} \, \leq k_{1}\parallel g \parallel_{\mathcal{H}^{1}(\Omega)}\,\, \text{and}\,\, \parallel g/a \parallel_{\mathcal{H}^{1}(\Omega)}\leq k_{2}\parallel g \parallel_{\mathcal{H}^{1}(\Omega)},$$ where the constants $k_{1}$ and $k_{2}$ do not depend on $g\in\mathcal{H}^{1}(\Omega)$. This implies that the functions $a$ and $1/a$ behave now as multipliers in the space $\mathcal{H}^{1}(\Omega)$. Furthermore, as long as $a\in \mathcal{C}^{1}(S)$, then $\partial_{n}a$ is also a multiplier. 
\end{remark}

The volume parametrix-based Newton-type potential and the remainder potential are respectively defined, for $y\in\mathbb R^3$, as
\begin{align*}
\mathcal{P}\rho(y)&:=\displaystyle\int_{\Omega} P(x,y)\rho(x)\hspace{0.25em}dx \\
\mathcal{R}\rho(y)&:=\displaystyle\int_{\Omega} R(x,y)\rho(x)\hspace{0.25em}dx.
\end{align*}

 The parametrix-based single layer and double layer  surface potentials are defined for $y\in\mathbb R^3:y\notin S $, as 
 \begin{align}
     \label{ch4SL}
V\rho(y)&:=-\int_{S} P(x,y)\rho(x)\hspace{0.25em}dS(x),\\
\label{ch4DL}
W\rho(y)&:=-\int_{S} T_{x}^{+}P(x,y)\rho(x)\hspace{0.25em}dS(x).
 \end{align}

We also define the following pseudo-differential operators associated with direct values of the single and  double layer potentials and with their conormal derivatives, for $y\in S$,
\begin{align}
\mathcal{V}\rho(y)&:=-\int_{S} P(x,y)\rho(x)\hspace{0.25em}dS(x),\nonumber \\
\mathcal{W}\rho(y)&:=-\int_{S} T_{x}P(x,y)\rho(x)\hspace{0.25em}dS(x),\nonumber\\
\mathcal{W'}\rho(y)&:=-\int_{S} T_{y}P(x,y)\rho(x)\hspace{0.25em}dS(x),\nonumber\\
\mathcal{L}^{\pm}\rho(y)&:=T_{y}^{\pm}{W}\rho(y)\nonumber.
\end{align}

The operators $\mathcal P, \mathcal R, V, W, \mathcal{V}, \mathcal{W}, \mathcal{W'}$ and $\mathcal{L}$ can be expressed in terms the volume and surface potentials and operators associated with the Laplace operator, as follows \begin{align}
\mathcal{P}\rho&=\mathcal{P}_{\Delta}\left(\dfrac{\rho}{a}\right),\label{ch4relP}\\
\mathcal{R}\rho&=\nabla\cdot\left[\mathcal{P}_{\Delta}(\rho\,\nabla \ln a)\right]-\mathcal{P}_{\Delta}(\rho\,\Delta \ln a),\label{ch4relR}\\
V\rho &= V_{\Delta}\left(\dfrac{\rho}{a}\right),\label{ch4relSL}\\
\mathcal{V}\rho &= \mathcal{V}_{\Delta} \left( \dfrac{\rho}{a}\right),\label{ch4relDVSL}\\
W\rho &= W_{\Delta}\rho -V_{\Delta}\left(\rho\frac{\partial \ln a}{\partial n}\right),\label{ch4relDL}\\
\mathcal{W}\rho &= \mathcal{W}_{\Delta}\rho -\mathcal{V}_{\Delta}\left(\rho\frac{\partial \ln a}{\partial n}\right),\label{ch4relDVDL}\\
\mathcal{W}'\rho &= a \mathcal{W'}_{\Delta}\left(\dfrac{\rho}{a}\right),\label{ch4relTSL}\\
\mathcal{L}^{\pm}\rho &= \widehat{\mathcal{L}}\rho - aT^{\pm}_\Delta V_{\Delta}\left(\rho\frac{\partial \ln a}{\partial n}\right),
\label{ch4relTDL}\\
\widehat{\mathcal{L}}\rho &:= a\mathcal{L}_{\Delta}\rho.\label{ch4hatL}
\end{align}
The symbols with the subscript $\Delta$ denote the analogous operators for the constant coefficient case, $a\equiv 1$. Furthermore, by the Lyapunov-Tauber theorem (cf. \cite{tauber, tauber2} and more references therein), $\mathcal{L}_{\Delta}^{+}\rho = \mathcal{L}_{\Delta}^{-}\rho = \mathcal{L}_{\Delta}\rho$.

These relations can be exploited to obtain mapping properties of the parametrix based surface and volume potentials taking into account those mapping properties already known for the analogous surface and volumen potentials constructed with the fundamental solution of the Laplace equation. 

One of the main differences with respect the bounded domain case is that the integrands of the operators $V$, $W$, $\mathcal{P}$ and $\mathcal{R}$ and their corresponding direct values and conormal derivatives do not always belong to $L^{1}$. In these cases, the integrals should be understood as the corresponding duality forms (or their their limits of these forms for the infinitely smooth functions, existing due to the density in corresponding Sobolev spaces).

\begin{theorem}\label{ch5thmapVW} Suppose that Condition \ref{cond1} holds. Then, the operators
\begin{align*}
V:&H^{-1/2}(S) \longrightarrow \mathcal{H}^{1}(\Omega),\\
W:&H^{1/2}(S) \longrightarrow \mathcal{H}^{1}(\Omega)
\end{align*} are continuous.
\end{theorem}
\begin{proof}
Let us consider a function $g \in H^{-1/2}(S)$, then $\dfrac{g}{a}$ also belongs to $H^{-1/2}(S)$ in virtue of Remark \ref{ch5remmult} and Condition \ref{cond1}. Then, relation \eqref{ch4relSL} along with the mapping property $V_{\Delta}:H^{-1/2}(S)\longrightarrow \mathcal{H}^{1}(\Omega;\Delta)$, cf. \cite[Theorem 4.1]{exterior}; it is clear that $Vg = V_{\Delta}\left(g/a\right)\in \mathcal{H}^{1}(\Omega;\Delta)$ what implies $Vg\in \mathcal{H}^{1}(\Omega)$. 

Let us prove now the result for the operator $W$. If $g \in H^{1/2}(S)$, then $\partial_{n}(\ln a)g$ also belongs to $H^{1/2}(S)$ in virtue of Remark \ref{ch5remmult} and Condition \ref{cond1}. Then, relation \eqref{ch4relDL} along with the mapping properties $V_{\Delta}:H^{-1/2}(S)\longrightarrow \mathcal{H}^{1}(\Omega;\Delta)$ and $W_{\Delta}:H^{1/2}(S)\longrightarrow \mathcal{H}^{1}(\Omega;\Delta)$ imply that $Wg\in \mathcal{H}^{1}(\Omega;\Delta)$ from where it follows that $Wg\in \mathcal{H}^{1}(\Omega)$.

\end{proof}

\begin{corollary} \label{ch5thmapVWPRc} The following operators are continuous under Condition \ref{cond1} and \eqref{deltacond},
\begin{align}
V&:H^{-1/2}(S) \longrightarrow \mathcal{H}^{1,0}(\Omega; \mathcal{A})\label{ch5opVcont},\\
W&: H^{1/2}(S) \longrightarrow \mathcal{H}^{1,0}(\Omega; \mathcal{A})\label{ch5opWcont}.
\end{align}
\end{corollary}

\begin{proof} Let us prove first the mapping property \eqref{ch5opVcont}. 

From  Theorem \ref{ch5thmapVW}, we have that 
$Vg\in \mathcal{H}^{1}(\Omega)$ for some $g\in H^{-1/2}(S)$. Hence, it suffices to prove that $Vg\in L^{2}(\omega;\Omega)$. 

Differentiating using the product rule, we can write
\begin{equation}\label{ch5cordA}
\mathcal{A}h = \nabla a \nabla h + a \Delta h.
\end{equation}

Taking into account relation \eqref{ch4relSL}
and applying \eqref{ch5cordA} to $h=V_{\Delta}(g/a)$, we get 
\begin{equation}
\mathcal{A}V_{\Delta}\left(\dfrac{g}{a}\right) = \sum_{i=1}^{3}\dfrac{\partial a}{\partial y_{i}} \dfrac{\partial V_{\Delta}}{\partial y_{i}} \left(\dfrac{g}{a}\right) + a \Delta V_{\Delta}\left(\dfrac{g}{a}\right) = \sum_{i=1}^{3}\dfrac{\partial a}{\partial y_{i}} \dfrac{\partial V_{\Delta}}{\partial y_{i}} \left(\dfrac{g}{a}\right)=\nabla a\nabla V(g).
\end{equation}
By virtue of the mapping property for the operator $V$ provided by Theorem \ref{ch5thmapVW}, the last term belongs to $L^{2}(\omega;\Omega)$ due to the fact that  $V_{\Delta}(g/a)=Vg\in \mathcal{H}^{1}(\Omega)$, and thus its derivatives belong to $L^{2}(\omega;\Omega)$. The term  $\nabla a$ acts as a multiplier in the space $L^{2}(\omega;\Omega)$ due to Condtion \ref{cond1}. On the other hand, the term $a\Delta V_{\Delta}(g/a)$ vanishes on $\Omega$ since $V_{\Delta}(\cdot)$ is the single layer potential for the Laplace equation, i.e., $V_{\Delta}(g/a)$ is a harmonic function. This, completes the proof for the operator $V$. 

The proof for the operator $W$ follows from a similar argument. 
\end{proof}

\begin{cond}\label{deltacond} In addition to Condition \ref{cond0} and Condition \ref{cond1}, we will sometimes need the following condition:
\begin{equation}\label{ch5conda(x)3}
\omega^{2}\Delta a\in L^{\infty}(\Omega).
\end{equation}
\end{cond}

\begin{remark}
Note as well that due to Condition \ref{cond0} and the continuity of the function $\ln a$, the components of $\nabla (\ln a)$ and $\Delta (\ln a)$ are bounded as well.
\end{remark}

\begin{theorem}\label{ch5thmapPPRc} The following operators are continuous under Condition \ref{cond1},
\begin{align}
\textbf{P}&: \mathcal{H}^{-1}(\mathbb{R}^{3})\longrightarrow \mathcal{H}^{1}(\mathbb{R}^{3}),\label{ch5opPcontbf}\\
\textbf{R}&: L^{2}(\omega^{-1};\mathbb{R}^{3}) \longrightarrow \mathcal{H}^{1}(\mathbb{R}^{3}),\label{ch5opRcontbf}\\
\mathcal{P}&: \widetilde{\mathcal{H}}^{-1}(\Omega)\longrightarrow \mathcal{H}^{1}(\mathbb{R}^{3})\label{ch5opPcont}.
\end{align}
\end{theorem}

\begin{proof} 
 Let $g\in \mathcal{H}^{-1}(\mathbb{R}^{3})$. Then, by virtue of the relation \eqref{ch4relP}
$\textbf{P}g = \textbf{P}_{\Delta}(g/a)$. Since Condition \ref{ch5conda(x)2} holds, $(g/a)\in \mathcal{H}^{-1}(\mathbb{R}^{3})$ and thefore the continuity of the operator $\textbf{P}$ follows from  the continuity of $\textbf{P}_{\Delta}:\mathcal{H}^{-1}(\mathbb{R}^{3})\longrightarrow \mathcal{H}^{1}(\mathbb{R}^{3})$, which at the same time implies the continuity of the operator \eqref{ch5opPcont}, see \cite[Theorem 4.1]{exterior} and more references therein.

Let us prove now the continuity of the operator $\textbf{R}$. Due to the second condition in \eqref{ch5conda(x)2}, the components of $\nabla a\in L^{2}(\mathbb{R}^{3})$ behave as multipliers in the space $L^{2}(\omega^{-1}; \mathbb{R}^{3})$. Let $g\in L^{2}(\omega^{-1}; \mathbb{R}^{3})$,  then the relation \eqref{ch4relR} applies and gives
\begin{align}
\mathbf{R}g(y)&=-\nabla\cdot \mathbf{P}_{\Delta}(g\cdot \nabla (\ln a))(y)\nonumber=-\sum_{i=1}^{3}\dfrac{\partial}{\partial y_{i}}\mathbf{P}_{\Delta}\left(g\cdot \dfrac{\partial(\ln a)}{\partial x_{i}}\right)(y)\nonumber\\
&=-\sum_{i=1}^{3}\mathbf{P}_{\Delta}\left[\dfrac{\partial}{\partial x_{i}}\left(g\cdot \dfrac{\partial(\ln a)}{\partial x_{i}}\right)\right](y):=-\mathbf{P}_{\Delta}g^{*}(y).
\end{align}
In this case, $g^{*}\in \mathcal{H}^{-1}(\mathbb{R}^{3})$ as a result of a similar argument as in \cite[Theorem 4.1]{exterior}. Here, $\nabla \ln a $ is multipliers under Condition \ref{ch5conda(x)2} in the space $\mathcal{H}^{-1}(\mathbb{R}^{3})$.

Since the operator ${\mathbf{P}_{\Delta}:\mathcal{H}^{-1}(\mathbb{R}^{3})\longrightarrow \mathcal{H}^{1}(\mathbb{R}^{3})}$ is continuous, the operator ${\mathbf{R}: L^{2}(\omega^{-1}; \mathbb{R}^{3})\longrightarrow \mathcal{H}^{1}(\mathbb{R}^{3})}$ is also continuous. \end{proof}
%

\begin{theorem} \label{ch5thmapPRH10} The following operators are continuous under Condition \ref{ch5conda(x)2} and \eqref{ch5conda(x)3},
\begin{align}
\mathcal{P}&: L^{2}(\omega; \Omega) \longrightarrow \mathcal{H}^{1,0}(\mathbb{R}^{3}; \mathcal{A}),\label{ch5opPH10}\\
\mathcal{R}&: \mathcal{H}^{1}(\Omega) \longrightarrow \mathcal{H}^{1,0}(\Omega; \mathcal{A})\label{ch5opRH10}.
\end{align}
\end{theorem}

\begin{proof}
To prove the continuity of the operator \eqref{ch5opPH10}, we consider a function $g\in L^{2}(\omega; \Omega)$ and its extension by zero to $\mathbb{R}^{3}$ which we denote by $\widetilde{g}$. Clearly, $\widetilde{g}\in L^{2}(\omega; \mathbb{R}^{3})\subset \mathcal{H}^{-1}(\mathbb{R}^{3})$ and then $\mathcal{P}_{\Delta} g = \mathbf{P}_{\Delta}\widetilde{g} \in \mathcal{H}^{1}(\mathbb{R}^{3})$.  Bearing in mind that
\begin{equation*}
\mathcal{A}\mathcal{P}g(y) = g(y)+\sum_{i=1}^{3}\dfrac{\partial a(y)}{\partial y_{i}} \dfrac{\partial\mathcal{P}_{\Delta}}{\partial y_{i}}\left(\dfrac{g}{a}\right)(y),
\end{equation*}

under Condition \ref{ch5conda(x)2}, we conclude that $\mathcal{A}\mathcal{P}g(y)\in L^{2}(\omega,\Omega)$ and therefore $\mathcal{P}g\in \mathcal{H}^{1,0}(\Omega, \mathcal{A})$.

Finally, let us prove the continuity of the operator \eqref{ch5opRH10}. The continuity of the operator $\mathcal{R}:\mathcal{H}^{1}(\Omega)\longrightarrow \mathcal{H}^{1}(\Omega)$ follows from the continuous embedding ${\mathcal{H}^{1}(\Omega)\subset L^{2}(\omega^{-1};\Omega)}$ and the continuity of the operator \eqref{ch5opRcontbf}. Hence, we only need to prove that $\mathcal{A}\mathcal{R}g\in L^{2}(\omega;\Omega)$. For $g\in \mathcal{H}^{1}(\Omega)$ we have
\begin{align*}
\mathcal{A}\mathcal{R}g(y) = \dfrac{\partial a(y)}{\partial y_{i}}\dfrac{\partial \mathcal{R}g(y)}{\partial y_{i}} + a(y)\Delta \mathcal{R}g(y). 
\end{align*} 
As $\mathcal{R}g\in \mathcal{H}^{1}(\Omega)$, we only need to prove that $\Delta\mathcal{R}g(y)\in L^{2}(\omega ; \Omega)$. Using the relation \eqref{ch4relR}, we obtain that
\begin{align*}
\Delta \mathcal{R}g(y) &=\Delta \left[ -\nabla\cdot \mathcal{P}_{\Delta} (g \nabla (\ln a))\right]=  -\nabla\cdot \Delta\mathcal{P}_{\Delta} (g \nabla (\ln a))= -\nabla\cdot(g \nabla (\ln a)), 
\end{align*}
since $g\in\mathcal{H}^{1}(\Omega)$, then $g\in L^{2}(\omega, \Omega)$. $\nabla (\ln a)$ is a multiplier in the space $\mathcal{H}^{1}(\Omega)$ by virtue of the second condition in  \eqref{ch5conda(x)2}, then $(g \nabla \ln a)\in \mathcal{H}^{1}(\Omega)$. Consequently, ${-\nabla\cdot(g \nabla \ln a)\in L^{2}(\omega; \Omega)}$ by virtue of Condition \ref{deltacond}, from where it follows the result. 
\end{proof}

%
%

\section{Third Green identities and integral relations}
 Applying the second Green identity \eqref{ch5secondgreen}, with $v=P(x,y)$ and any distribution $u\in \mathcal{H}^{1,0}(\Omega;\mathcal{A})$ in $\Omega$, we obtain the third Green identity (\textit{integral representation formula}) for the function $u\in \mathcal{H}^{1,0}(\Omega ; \mathcal{A})$:
\begin{equation}\label{ch5green3}
u+\mathcal{R}u-VT^{+}u+W\gamma^{+}u=\mathcal{P}\mathcal{A}u,\hspace{1em}\text{in}\hspace{0.5em}\Omega.
\end{equation}

If $u\in \mathcal{H}^{1,0}(\Omega; \mathcal{A})$ is a solution of the PDE \eqref{ch5BVP1}, then, from \eqref{ch5green3}, we obtain
\begin{equation}\label{ch53GV}
u+\mathcal{R}u-VT^{+}(u)+W\gamma^{+}u=\mathcal{P}f,\hspace{0.5em}\text{in}\hspace{0.5em}\Omega. 
\end{equation}

Taking the trace and the conormal derivative of \eqref{ch53GV}, we obtain integral representation formulae for the trace and traction of $u$ respectively:
\begin{align}
    \label{ch53GG}
\dfrac{1}{2}\gamma^{+}u+\gamma^{+}\mathcal{R}u-\mathcal{V}T^{+}u+\mathcal{W}\gamma^{+}u&=\gamma^{+}\mathcal{P}f,\hspace{0.5em}\text{on}\hspace{0.5em}S,\\\label{ch53GT}
\dfrac{1}{2}T^{+}u+T^{+}\mathcal{R}u-\mathcal{W'}T^{+}u+\mathcal{L}^{+}\gamma^{+}u&=T^{+}\mathcal{P}f,\hspace{0.5em}\text{on}\hspace{0.5em}S.
\end{align}

For some distributions $f, \Psi$ and $\Psi$, we consider a more indirect integral relation associated with the third Green identity \eqref{ch53GV}
\begin{equation}\label{ch5G3ind}
u+\mathcal{R}u-V\Psi+W\Phi=\mathcal{P}f,\hspace{0.5em}\text{in}\hspace{0.5em}\Omega.
\end{equation}

\begin{lemma}\label{ch5L1} Let $u\in \mathcal{H}^{1}(\Omega)$, $f\in L_{2}(\omega;\Omega)$, $\Psi\in H^{-1/2}(S)$ and $\Phi\in H^{1/2}(S)$, satisfying the relation \eqref{ch5G3ind}. Let conditions \eqref{ch5conda(x)2} and \eqref{ch5conda(x)3} hold.  Then $u\in \mathcal{H}^{1,0}(\Omega, \mathcal{A})$, solves the equation $\mathcal{A}u=f$ in $\Omega$ and the following identity is satisfied
\begin{equation}\label{ch5lema1.0}
V(\Psi- T^{+}u) - W(\Phi- \gamma^{+}v) = 0,\hspace{0.5em}\text{in}\hspace{0.5em}\Omega.
\end{equation}
\end{lemma}

\begin{proof}
To prove that $u\in \mathcal{H}^{1,0}(\Omega;\mathcal{A})$, taking into account that by hypothesis $u\in \mathcal{H}^{1}(\Omega)$, so there is only left to prove that $\mathcal{A}u \in L^{2}(\omega; \Omega)$. 
Firstly we write the operator $\mathcal{A}$ as follows:
\begin{align*}
&\mathcal{A}(x)[u(x)]=\Delta(au)(x) - \sum_{i=1}^{3}\dfrac{\partial}{\partial x_{i}}\left(u\left(
\dfrac{\partial a(x)}{\partial x_{i}}\right)\right).
\end{align*}

It is easy to see that the second term belongs to $L^{2}(\omega; \Omega)$. Keeping in mind Remark \ref{ch5remmult} and the fact that $u\in \mathcal{H}^{1}(\Omega)$, then we can conclude that the term $u\nabla a\in \mathcal{H}^{1}(\Omega)$ since due to the second condition in \eqref{ch5conda(x)2} $\nabla a$ is a multiplier in the space $\mathcal{H}^{1}(\Omega)$ and therefore $\nabla(u\nabla a)\in L^{2}(\omega; \Omega)$.

Now, we only need to prove that $\Delta (au)\in L^{2}(\omega;\Omega)$. To prove this we look at the relation \eqref{ch5G3ind} and we put $u$ as the subject of the formula. Then, we use the potential relations \eqref{ch4relP},  \eqref{ch4relSL} and \eqref{ch4relDL} 
\begin{equation}\label{ch5lema1.1}
u=\mathcal{P}f-\mathcal{R}u+V\Psi-W\Phi=\mathcal{P}_{\Delta}\left(\dfrac{f}{a}\right)-\mathcal{R}u+V_{\Delta}\left(\dfrac{\Psi}{a}\right)-W_{\Delta}\Phi+V_{\Delta} \left(\dfrac{\partial(\ln(a))}{\partial n}\Phi\right)
\end{equation}

In virtue of the Theorem \ref{ch5thmapPRH10}, $\mathcal{R}u\in L^{2}(\omega;\Omega)$. Moreover, the terms in previous expression depending on $V_{\Delta}$ or $W_{\Delta}$ are harmonic functions and $\mathcal{P}_{\Delta}$ is the newtonian potential for the Laplacian, i.e. $\Delta\mathcal{P}_{\Delta}\left(\dfrac{f}{a}\right)=\dfrac{f}{a}$. Consequently, applying the Laplacian operator in both sides of \eqref{ch5lema1.1}, we obtain:
\begin{equation}\label{ch5lema1.2}
\Delta u = \dfrac{f}{a}-\Delta\mathcal{R}u. 
\end{equation}

 Thus, $\Delta u \in L^{2}(\omega; \Omega)$ from where it immediately follows that $\Delta (au)\in L^{2}(\omega; \Omega)$. Hence $u\in \mathcal{H}^{1,0}(\Omega;\mathcal{A})$. The rest of the proof is equivalent to \cite[Lemma 5.1]{carloscomp}.
\end{proof}

The proof of the following statement is the counterpart of \cite[Lemma 5.2]{carloscomp} for exterior domains. The proof follows from the invertibility of the operator $\mathcal{V}_{\Delta}$, see \cite[Corollary 8.13]{mclean}. 

\begin{lemma}\label{ch5L2}
Let $\Psi^{*}\in H^{-1/2}(S)$. If
\begin{equation}\label{ch5lema2i}
V\Psi^{*}(y) = 0, \hspace{0.5em}y\in\Omega,
\end{equation}
then $\Psi^{*}(y) = 0$.
\end{lemma}

\begin{proof}
Take the trace of \eqref{ch5lema2i} and relation \eqref{ch4relSL}, to obtain 
\begin{equation}\label{ch5lema2i2}
\mathcal{V}\Psi^{*}(y) = \mathcal{V}_{\Delta}\left(\dfrac{\Psi^{*}}{a}\right)(y)=0, \hspace{0.5em}y\in\,\, S.
\end{equation}
Then, applying \cite[Corollary 8.13]{mclean}, we obtain that the equation \eqref{ch5lema2i2} is uniquely solvable. Hence, $\Psi^{*}(y) = 0$.
\end{proof}
\section{BDIES}
In this section, we will derive a system of boundary domain integral equations formally segregated from the solution $u$ of the BVP \eqref{ch5BVP1}-\eqref{ch5BVPN}, following a similar approach as in \cite[Section 5]{mikhailov1}. Consequently, we introduce $\Phi_{0} \in H^{1/2}(S)$ and $\Psi_{0}\in H^{-1/2}(S)$ as continuous fixed extensions to $S$ of the functions $\phi_{0}\in H^{1/2}(S_{D}) $ and $\psi_{0}\in H^{-1/2}(S_{N})$. Moreover, let $\phi\in\widetilde{H}^{1/2}(S_{N})$ and $\psi\in\widetilde{H}^{-1/2}(S_{D})$ be arbitrary functions formally segregated from $u$. Then, make 
\begin{equation}\label{ch5phipsi}
\gamma^{+}u = \Phi_{0} + \phi, \quad\quad\quad T^{+}u = \Psi_{0} + \psi,\quad \text{on}\quad S;
\end{equation}
in the three third Green identities \eqref{ch53GV}-\eqref{ch53GT} to obtain the to obtain the following BDIES (M12)
\begin{subequations}
\begin{align}
u+\mathcal{R}u-V\psi+W\phi&=F_{0},\hspace{0.5em}\text{in}\hspace{0.1em}\,\,\,\Omega,\label{ch5SM12v}\\
\dfrac{1}{2}\phi+\gamma^{+}\mathcal{R}u-\mathcal{V}\psi+\mathcal{W}\phi&=\gamma^{+}F_{0}-\Phi_{0},\label{ch5SM12g}\hspace{0.5em}\text{on}\hspace{0.1em}\,\,\,S.
\end{align}
\end{subequations}
In what follows, we will denote by $\mathcal{X}$ the vector of unknown functions
 \[\mathcal{X} = (u,\psi,\phi)^{\top}\in \mathbb{H}:=\mathcal{H}^{1,0}(\Omega; \mathcal{A})\times \widetilde{H}^{-1/2}(S_{D})\times \widetilde{H}^{1/2}(S_{N})\subset \mathbb{X}\]
where $\mathbb{X} := \mathcal{H}^{1}(\Omega)\times \widetilde{H}^{-1/2}(S_{D})\times \widetilde{H}^{1/2}(S_{N}).$
We will denote by $\mathcal{M}^{12}$ the matrix operator that defines the system $(M12)$:
\begin{equation}
   \mathcal{M}^{12}=
  \left[ {\begin{array}{ccc}
   I+\mathcal{R} & -V & +W \\
   \gamma^{+}\mathcal{R} & -\mathcal{V} & \dfrac{1}{2}I + \mathcal{W} 
  \end{array} } \right],
\end{equation}
and by $\mathcal{F}^{12}$ the right hand side of the system $\mathcal{F}^{12} = [ \,F_{0},\,\, \gamma^{+}F_{0}-\Phi_{0}\,]^{\top}.$

Using this notation, the system (M12) can be rewritten in terms of matrix notation as $\mathcal{M}^{12}\mathcal{X} = \mathcal{F}^{12}$.

If Condition \ref{ch5conda(x)2} and Condition \ref{ch5conda(x)3} hold, then, due to the mapping properties of the potentials, $\mathcal{F}^{12}\in \mathbb{F}^{12}\subset \mathbb{Y}^{12}$, while operators $\mathcal{M}^{12}:\mathbb{H}\rightarrow \mathbb{F}^{12}$ and $\mathcal{M}^{12}:\mathbb{X}\rightarrow \mathbb{Y}^{12}$ are continuous. Here, we denote
\begin{align*}
\mathbb{F}^{12}&:=\mathcal{H}^{1,0}(\Omega, \mathcal{A})\times H^{1/2}(S),&
\mathbb{Y}^{12}&:=\mathcal{H}^{1}(\Omega)\times H^{1/2}(S).
\end{align*}

The following result shows that the BDIES (M12) is equivalent to the original mixed BVP \eqref{ch5BVP1}-\eqref{ch5BVPN}. 

\begin{theorem}\label{ch5equivalence}
Let $f\in L_{2}(\omega; \Omega)$, let $\Phi_{0}\in H^{-1/2}(S)$ and let $\Psi_{0}\in H^{-1/2}(S)$ be some fixed extensions of $\phi_{0}\in H^{1/2}(S_{D})$ and $\psi_{0}\in H^{-1/2}(S_{N})$, respectively. Let Condition \ref{ch5conda(x)2} and Condition \ref{ch5conda(x)3} hold. Then, 
\begin{enumerate}
\item[i)] if some $u\in \mathcal{H}^{1,0}(\Omega;\mathcal{A})$ solves the BVP \eqref{ch5BVP1}-\eqref{ch5BVPN}, then the triplet $(u, \psi,\phi )^{\top}\in \mathcal{H}^{1,0}(\Omega;\mathcal{A})\times\widetilde{H}^{-1/2}(S_{D})\times\widetilde{H}^{1/2}(S_{N})$ where
\[
\phi=\gamma^{+}u-\Phi_{0},\hspace{2em}\psi=T^{+}u-\Psi_{0}, \quad \quad \text{on}\,\,\, S,
\]
solves the BDIES (M12). 
\item[ii)] If a triple $(u, \psi, \phi )^{\top}\in \mathcal{H}^{1,0}(\Omega;\mathcal{A})\times\widetilde{H}^{-1/2}(S_{D})\times\widetilde{H}^{1/2}(S_{N})$ solves the BDIES (M12), then this solution is unique. Furthermore, $u$ solves the BVP \eqref{ch5BVP1}-\eqref{ch5BVPN} and the functions $\psi, \phi$ satisfy
\begin{equation}\label{ch5eqcond}
\phi=\gamma^{+}u-\Phi_{0},\hspace{2em}\psi=T^{+}u-\Psi_{0}, \quad \quad \text{on}\,\,\, S.
\end{equation} 
\end{enumerate}
\end{theorem}

\begin{proof}
The proof of item $i)$ automatically follows from the derivation of the BDIES (M12).

Let us prove now item $ii)$. Let the triple $(u, \psi,\phi )^{\top}\in (u, \psi, \phi )^{\top}\in \mathcal{H}^{1,0}(\Omega;\mathcal{A})\times\widetilde{H}^{-1/2}(S_{D})\times\widetilde{H}^{1/2}(S_{N})$ solve the BDIE system. Taking the trace of the equation \eqref{ch5SM12v} and substract it from the equation \eqref{ch5SM12g}, we obtain
\begin{equation}\label{ch5M12a1}
\phi=\gamma^{+}u-\Phi_{0}, \hspace{1em} \text{on}\hspace{0.5em}S.
\end{equation}
This means that the first condition in \eqref{ch5eqcond} is satisfied. 
Now, restricting equation \eqref{ch5M12a1} to $S_{D}$, we observe that $\phi$ vanishes as $supp(\phi)\subset S_{N}$. Hence, $\phi_{0}=\Phi_{0}=\gamma^{+}u$ on $S_{D}$ and consequently, the Dirichlet condition of the BVP \eqref{ch5BVPD} is satisfied. 

We proceed using the Lemma \ref{ch5L1} in equation  \eqref{ch5SM12v}, with $\Psi=\psi + \Psi_{0}$ and $\Phi = \phi + \Phi_{0}$ which implies that $u$ is a solution of the equation \eqref{ch5BVP1} and also the following equality:
\begin{equation}\label{ch5M12a2}
V(\Psi_{0}+\psi - T^{+}u) - W(\Phi_{0} + \phi -\gamma^{+}u) = 0 \text{ in } \Omega.
\end{equation}
In virtue of \eqref{ch5M12a1}, the second term of the previous equation vanishes. Hence,
\begin{equation}\label{ch5M12a3}
V(\Psi_{0}+\psi - T^{+}u)= 0, \quad \text{ in } \Omega.
\end{equation}
Now, in virtue of Lemma \ref{ch5L2} we obtain
\begin{equation}\label{ch5M12a4}
\Psi_{0} + \psi - T^{+}u = 0,\quad\text{ on } S.
\end{equation}
Since $\psi$ vanishes on $S_{N}$, we can conclude $\Psi_{0}=\psi_{0}$ on $S_{N}$. Consequently, equation \eqref{ch5M12a4} implies that $u$ satisfies the Neumann condition \eqref{ch5BVPN}. \end{proof}

\section{Representation Theorems and Invertibility}
In this section, we aim to prove the invertibility of the operator $\mathcal{M}^{12}:\mathbb{H}\rightarrow \mathbb{F}^{12}$ by showing first that the arbitrary right hand side $\mathbb{F}^{12}$ from the respective spaces can be represented in terms of the parametrix-based potentials and using then the equivalence theorems. 

The following result is the counterpart of \cite[Lemma 7.1]{exterior} for the new parametrix $P^{x}(x,y)$. The analogous result for bounded domains can be found in \cite[Lemma 3.5]{mikhailov1}.

\begin{lemma}\label{ch5L3} For any function $\mathcal{F}_{*}\in \mathcal{H}^{1,0}(\Omega;\mathcal{A})$, there exists a unique couple $(f_{*}, \Psi_{*}) = \mathcal{C}\mathcal{F}_{*}\in L^{2}(\omega ; \Omega) \times H^{-1/2}(S)$ such that
\begin{equation}\label{ch5Lrep1}
\mathcal{F}_{*}(y)=\mathcal{P}f_{*}(y) + V\Psi_{*}(y), \quad y\in \Omega, 
\end{equation}
where $\mathcal{C}:\mathcal{H}^{1,0}(\Omega;\mathcal{A})\rightarrow L^{2}(\omega ; \Omega)\times H^{-1/2}(S)$ is a linear continuous operator. 
\end{lemma}

\begin{proof}
Let us assume that such functions $f_{*}$ and $\Psi_{*}$, satisfying \eqref{ch5Lrep1}, exist. Then, we aim to find expressions of these functions in terms of $\mathcal{F}_{*}$. Applying the potential relations \eqref{ch4relSL}, \eqref{ch4relP} to the equation \eqref{ch5Lrep1}, we obtain
\begin{equation}\label{ch5Lrep2}
\mathcal{F}_{*}(y)=\mathcal{P}_{\Delta}\left(\dfrac{f_{*}}{a}\right)(y)
 + V_{\Delta}\left(\dfrac{\Psi_{*}}{a}\right)(y), \quad y\in \Omega. 
\end{equation}
Applying the Laplace operator at both sides of the equation \eqref{ch5Lrep2}, we get
\begin{equation}\label{ch5Lrep3}
f_{*}= a\Delta \mathcal{F}_{*}.
\end{equation}
On the other hand, we can rewrite equation \eqref{ch5Lrep2}
as
\begin{equation}\label{ch5Lrep4}
V_{\Delta}\left(\dfrac{\Psi_{*}}{a}\right)(y) = Q(y), \quad y\in \Omega,
\end{equation}
where 
\begin{equation}\label{ch5Lrep5}
Q(y):=\mathcal{F}_{*}(y)- \mathcal{P}_{\Delta}\left( \Delta \mathcal{F}_{*}\right)(y).
\end{equation}
Now, we take the trace of \eqref{ch5Lrep4}
\begin{equation}\label{ch5Lrep6}
\mathcal{V}_{\Delta}\left(\dfrac{\Psi_{*}}{a}\right)(y) = \gamma^{+}Q(y), \quad y\in S.
\end{equation}
It is well known that the direct value operator of the single layer potential for the Laplace equation $\mathcal{V}_{\Delta}:H^{-1/2}(S)\longrightarrow H^{1/2}(S)$ is invertible (cf. e.g. \cite[Corollary 8.13]{mclean}). Hence, we obtain the following expresion for $\Psi_{*}$:
\begin{equation}\label{ch5Lrep7}
\Psi_{*}(y) = a\mathcal{V}^{-1}_{\Delta}\gamma^{+}Q(y), \quad y\in S.
\end{equation}
Relations \eqref{ch5Lrep3} and \eqref{ch5Lrep7} imply the uniqueness of the couple $(f_{*},\Psi_{*})$. 

Now, we just simply need to prove that the pair $(f_{*},\Psi_{*})$ given by \eqref{ch5Lrep7} and \eqref{ch5Lrep3} satisfies \eqref{ch5Lrep1}. For this purpose, let us note that the single layer potential operator, $V_{\Delta}(\Psi_{*}/a)$ with $\Psi_{*}$ given by \eqref{ch5Lrep7}, as well as $Q(y)$ given by \eqref{ch5Lrep5} are both harmonic functions. 
Since $Q(y)$ and $V_{\Delta}(\Psi_{*}/a)$ are two harmonic functions that coincide on the boundary due to \eqref{ch5Lrep6}, then they must be identical in the whole $\Omega$ due to the uniqueness of solution to the Dirichlet problem for the Laplace equation, see \cite[Theorem 3.1]{exterior}. As a consequence, \eqref{ch5Lrep4} is true which implies \eqref{ch5Lrep1}. Thus, relations \eqref{ch5Lrep3}, \eqref{ch5Lrep5} and \eqref{ch5Lrep7} give
\begin{equation}\label{ch5Lrep8}
(f_{*}, \Psi_{*}) = \mathcal{C}\mathcal{F}_{*} := (a\Delta \mathcal{F}_{*}, a\mathcal{V}^{-1}_{\Delta}\gamma^{+}[\mathcal{F}_{*} -\mathcal{P}_{\Delta}(a\Delta \mathcal{F}_{*})]).
\end{equation}
Since all the operators involved in the definition \eqref{ch5Lrep8} of the operator $\mathcal{C}$ are continuous and linear, the operator $\mathcal{C}$ is also continuous and linear.
\end{proof}

\begin{corollary}\label{ch5cor71} Let \[(\mathcal{F}_{0},\mathcal{F}_{1})\in\mathcal{H}^{1,0}(\Omega ; \mathcal{A})\times H^{1/2}(\partial \Omega).\] Then there exists a unique triplet $(f_{*}, \Psi_{*},\Phi_{*})$ such that  $(f_{*}, \Psi_{*},\Phi_{*}) = \mathcal{C}_{*}(\mathcal{F}_{0},\mathcal{F}_{1})^{\top}$, where $\mathcal{C}_{*}:\mathcal{H}^{1,0}(\Omega, \mathcal{A})\times H^{1/2}(S)\rightarrow L^{2}(\omega;\Omega)\times H^{-1/2}(S) \times H^{1/2}(S)$ is a linear an bounded operator and $(\mathcal{F}_{0},\mathcal{F}_{1})$ are given by
\begin{align}
\mathcal{F}_{0} &= \mathcal{P}f_{*} +V\Psi_{*}-W\Phi_{*}\quad\text{in}\hspace{0.5em}\Omega\label{ch5cor71a}\\
\mathcal{F}_{1} &=\gamma^{+}\mathcal{F}_{0}-\Phi_{*}\quad\text{on}\hspace{0.5em}S\label{ch5cor71b}
\end{align}
\end{corollary}

\begin{proof}
Taking $\Phi_{*} = \gamma^{+}\mathcal{F}_{0}-\mathcal{F}_{1}$ and applying the previous lemma to $\mathcal{F}_{*} = \mathcal{F}_{0} + W\Phi_{*}$ we prove existence of the representation \eqref{ch5cor71a} and \eqref{ch5cor71b}. The uniqueness follows from the homogenenous case when $\mathcal{F}_{0}=\mathcal{F}_{1}=0$. Then, \eqref{ch5cor71b} implies $\Phi_{*}=0$ and consequently, by $\eqref{ch5cor71a}$ and Lemma \ref{ch5L3}, we get $\Psi_{*}=0$ and $f_{*}=0$.
\end{proof}

We are ready to prove one of the main results for the invertibility of the matrix operator of the BDIES (M12). 

\begin{theorem}If conditions \eqref{ch5conda(x)2} and \eqref{ch5conda(x)3} hold, then the following operator is continuous and continuously invertible:
\begin{align}
\mathcal{M}^{12}&:\mathbb{H}\rightarrow\mathbb{F}^{12}\label{ch5invM12RL}
\end{align}
\end{theorem}
\begin{proof}
In order to prove the invertibility of the operator $\mathcal{M}^{12}:\mathbb{H}\longrightarrow \mathbb{F}^{12}$, we apply the Corollary \ref{ch5cor71} to any right-hand side $\mathcal{F}^{12}\in \mathbb{F}^{12}$ of the equation $\mathcal{M}^{12}\mathcal{U}=\mathcal{F}^{12}$. Thus, $\mathcal{F}^{12}$ can be uniquely represented as $(f_{*}
, \Psi_{*}, \Phi_{*})^{\top} = \mathcal{C}_{*}\mathcal{F}^{12}$ as in \eqref{ch5cor71a}-\eqref{ch5cor71b} where $\mathcal{C}_{*}:\mathbb{F}^{12}\longrightarrow L^{2}(\omega;\Omega)\times H^{-1/2}(S) \times H^{1/2}(S)$ is continuous. 

In virtue of the equivalence theorem for the system (M12), Theorem \ref{ch5equivalence},  and the invertibility theorem for the boundary value problem with mixed boundary conditions, Theorem \ref{ch5thinv0}, the matrix equation $\mathcal{M}^{12}\mathcal{U}=\mathcal{F}^{12}$ has a solution $\mathcal{U} = (\mathcal{M}^{12})^{-1}\mathcal{F}^{12}$
where the operator $(\mathcal{M}^{12})^{-1}$, is given by expressions 
\begin{equation}
u = \mathcal{A}^{-1}_{M}[f_{*}, r_{S_{D}}\Phi_{*}, r_{S_{N}}\Psi_{*}], \quad \psi = T^{+}u -\Psi_{*}, \quad \phi= \gamma^{+}u - \Phi_{*},
\end{equation} where $(f_{*}
, \Psi_{*}, \Phi_{*})^{\top} = \mathcal{C}_{*}\mathcal{F}^{12}$. Consequently, the operator $(\mathcal{M}^{12})^{-1}$  is a continuous right inverse to the operator \eqref{ch5invM12RL}. Moreover, the operator $(\mathcal{M}^{12})^{-1}$  results to be a double sided inverse in virtue of the injectivity implied by Theorem \ref{ch5equivalence}.
\end{proof}

\section{Fredholm properties and Invertibility}
In this section, we are going to benefit from the compactness properties of the operator $\mathcal{R}$ to prove invertibility of the operator $\mathcal{M}^{12}: \mathbb{X} \rightarrow \mathbb{Y}^{12}$. This invertibility result is more general than the one presented in the previous section. The price to pay is imposing an additional condition on the variable coefficient. 

Unlike as in the bounded case, see  similar to \cite[Section 7.2]{exterior}, the Rellich compact embedding theorem cannot be applied as $\Omega$ is a bounded domain. Still, we can overcome this obstacle by decomposing the operator $\mathcal{R}$ into the sum of two operators: one which can be made arbitrarily small and the other one will be compact. Then, we shall simply make use of the Fredholm alternative to prove the invertibility of the matrix operator that defines the (M12) BDIES. However, we can only split the operator $\mathcal{R}$ if the PDE satisfies the additional condition
\begin{equation}\label{ch5conda(x)4}
\lim_{\vert x\vert \rightarrow \infty} \omega(x)\nabla a(x) = 0.
\end{equation}

\begin{lemma} Let conditions \eqref{ch5conda(x)2} and \eqref{ch5conda(x)4} hold. Then, for any $\epsilon>0$ the operator $\mathcal{R}$ can be represented as $\mathcal{R}=\mathcal{R}_{s} + \mathcal{R}_{c}$, where $\parallel \mathcal{R}_{s} \parallel_{\mathcal{H}^{1}(\Omega)}< \epsilon$, while $\mathcal{R}_{c}: \mathcal{H}^{1}(\Omega) \rightarrow \mathcal{H}^{1}(\Omega)$ is compact. 
\end{lemma}

\begin{proof}
Let $B(0,r)$ be the ball centered at $0$ with radius $r$ big enough such that $S\subset B_{r}$. Furthermore, let $\chi\in \mathcal{D}(\mathbb{R}^{3})$ be a cut-off function such that $\chi=1$ in $S\subset B_{r}$, $\chi = 0$ in $\mathbb{R}^{3}\smallsetminus B_{2r}$ and $0\leq \chi(x) \leq 1$ in $\mathbb{R}^{3}$. Let us define by $\mathcal{R}_{c}g:=\mathcal{R}(\chi g)$, $\mathcal{R}_{s}g:=\mathcal{R}((1-\chi)g)$. 

We will prove first that the norm of $\mathcal{R}_{s}$ can be made infinitely small. Let $g\in \mathcal{H}^{1}(\Omega)$, then
\begin{align*}
 &\parallel \mathcal{R}_{s}g \parallel_{\mathcal{H}^{1}(\Omega)}= \parallel \sum_{i=1}^{3}\mathcal{P}_{\Delta}\left[ \dfrac{\partial}{\partial x_{i}}\left(\sum_{i=1}^{3}\dfrac{\partial(\ln a)}{\partial x_{i}}(1-\chi)g\right)\right]\parallel_{\mathcal{H}^{1}(\Omega)}\leq k \parallel \mathcal{P}_{\Delta} \parallel_{\widetilde{\mathcal{H}}^{-1}(\Omega)},\\ 
&\text{with}\quad k:=\sum_{i=1}^{3}\parallel \dfrac{\partial}{\partial x_{i}}\left(\sum_{i=1}^{3}\dfrac{\partial(\ln a)}{\partial x_{i}}(1-\chi)g\right)\parallel_{\widetilde{\mathcal{H}}^{-1}(\Omega)}\,\, \leq \sum_{i=1}^{3} \parallel \dfrac{\partial(\ln a)}{\partial x_{i}}(1-\chi)g\parallel_{{L}^{2}(\Omega)}\\
 & \hspace{4em} \leq 3 \parallel g \parallel_{L^{2}(\omega^{-1};\Omega)}\parallel \omega \nabla a\parallel_{L^{\infty}(\mathbb{R}^{3}\smallsetminus B_{r})}\,\,\leq\,\, 3 \parallel g \parallel_{\mathcal{H}^{1}(\Omega)}\parallel \omega \nabla a\parallel_{L^{\infty}(\mathbb{R}^{3}\smallsetminus B_{r})}.
\end{align*} 
Consequently, we have the following estimate:
\begin{align*}
\parallel \mathcal{R}_{s}g \parallel_{\mathcal{H}^{1}(\Omega)} &\leq 3 \parallel g \parallel_{\mathcal{H}^{1}(\Omega)}\parallel \omega \nabla a\parallel_{L^{\infty}(\mathbb{R}^{3}\smallsetminus B_{r})}\parallel \mathcal{P}_{\Delta} \parallel_{\widetilde{\mathcal{H}}^{-1}(\Omega)}.
\end{align*}

Using the previous estimate is easy to see that when $\epsilon\rightarrow +\infty$ the norm $\parallel \mathcal{R}_{s}g \parallel_{\mathcal{H}^{1}(\Omega)}$ tends to $0$. Hence, the norm of the operator $\mathcal{R}_{s}$ can be made arbitrarily small. 

To prove the compactness of the operator $\mathcal{R}_{c}g:=\mathcal{R}(\chi g)$, we recall that ${supp(\chi)\subset \bar{B}(0,2r)}$. Then, one can express $\mathcal{R}_{c}g:= \mathcal{R}_{\Omega_{r}}([\chi g\vert_{\Omega_{r}}])$ where the operator $\mathcal{R}$ is defined now over $\Omega_{r}:=\Omega\cap B_{2r}$ which is a bounded domain. As the restriction operator $\vert_{\Omega_{r}}:\mathcal{H}^{1}(\Omega)\longrightarrow \mathcal{H}^{1}(\Omega_{r})$ is continuous, in virtue of Theorem \ref{ch5thmapPPRc}, the operator $\mathcal{R}_{c}g: L^{2}(\Omega_{r})\longrightarrow \mathcal{H}^{1}(\Omega_{r})$ is also continuous. Due to the boundedness of $\Omega_{r}$, we have $\mathcal{H}^{1}(\Omega_{r})= H^{1}(\Omega_{r})$ and thus the compactness of $\mathcal{R}_{c}g$ follows from the Rellich Theorem applied to the embedding $L^{2}(\Omega_{r})\subset H^{1}(\Omega_{r})$.
\end{proof}

\begin{corollary} Let conditions \eqref{ch5conda(x)2} and \eqref{ch5conda(x)4} hold. Then, the operator ${I+\mathcal{R}:\mathcal{H}^{1}(\Omega) \rightarrow \mathcal{H}^{1}(\Omega)}$ is Fredholm with zero index. 
\end{corollary}   

\begin{proof}
Using the previous Lemma, we have $\mathcal{R}=\mathcal{R}_{s} + \mathcal{R}_{c}$ so $\parallel \mathcal{R}_{s}\parallel < 1$ hence $I+\mathcal{R}_{s}$ is invertible. On the other hand $\mathcal{R}_{c}$ is compact and hence $I + \mathcal{R}_{s}$ a compact perturbation of the operator $I+\mathcal{R}$, from where it follows the result.
\end{proof}

\begin{theorem}
If conditions \eqref{ch5conda(x)2}, \eqref{ch5conda(x)3} and \eqref{ch5conda(x)4} hold, then the operator 
\begin{equation}
\mathcal{M}^{12}:\mathbb{X}\rightarrow \mathbb{Y}^{12},
\end{equation}
is continuously invertible. 
\end{theorem}

\begin{proof}
Let \[
   \mathcal{M}^{12}_{0}=
  \left[ {\begin{array}{ccc}
   I & -V & W \\
   0 & -\mathcal{V} & \dfrac{1}{2}I\\
  \end{array} } \right].
\]

Let $\mathcal{U} = (u,\psi,\phi)\in \mathbb{X}$ be a solution of the equation $\mathcal{M}_{0}^{12}\mathcal{U} = \mathcal{F}$, where ${\mathcal{F}  = (\mathcal{F}_{1}, \mathcal{F}_{2})\in \mathcal{H}^{1}(\Omega)\times H^{1/2}(S)}$. Then, $\mathcal{U}$ will also solve the following extended system
\begin{align}
u+W\phi - V\psi &= \mathcal{F}_{1}\quad\text{in}\hspace{0.5em} \Omega \nonumber ,\\
\dfrac{1}{2}\phi - \mathcal{V}\psi &=\mathcal{F}_{2}\quad\text{on}\hspace{0.5em}S,\label{ch5sysm21}\\
-r_{S_{D}}\mathcal{V}\psi &= r_{S_{D}}\mathcal{F}_{2}\quad\text{on}\hspace{0.5em}S_{D}\nonumber.
\end{align}

Furthermore, every solution of the system \eqref{ch5sysm21} will solve the equation $\mathcal{M}_{0}^{12}\mathcal{U} = \mathcal{F}$. 

The system \eqref{ch5sysm21} can be written also in matrix form as $\widetilde{\mathcal{M}}^{12}_{0}\mathcal{U}=\widetilde{\mathcal{F}}$ where $\widetilde{\mathcal{F}}$ denotes the right hand side and $\widetilde{\mathcal{M}}^{12}_{0}$ is defined as
\[
   \widetilde{\mathcal{M}}^{12}_{0}:=
  \left[ {\begin{array}{ccc}
   I & W & -V \\
   0 &  \dfrac{1}{2}I & -\mathcal{V}\\
   0 & 0  & -r_{S_{D}}\mathcal{V}\\
  \end{array} } \right].
\]

We note that the three diagonal operators:
\begin{align*}
I &: \mathcal{H}^{1}(\Omega) \longrightarrow \mathcal{H}^{1}(\Omega), \\
\dfrac{1}{2}I &: H^{1/2}(S) \longrightarrow H^{1/2}(S),\\
-r_{S_{D}}\mathcal{V} &: \widetilde{H}^{-1/2}(S_{D})\longrightarrow H^{1/2}(S_{D})
\end{align*}
are invertible, cf. \cite[Theorem 4.7]{carloscomp}. Hence, the operator $\widetilde{\mathcal{M}}^{12}_{0}$ which defines the system \eqref{ch5sysm21} is invertible. 

Now, let $\psi\in \widetilde{H}^{-1/2}(S_{D})
$ such that the third equation in the system \eqref{ch5sysm21} is satisfied. Then, solving $\phi$ from the second equation of the system, we get ${\phi=2(\mathcal{V}\psi + \mathcal{F}_{2})\in \widetilde{H}^{1/2}(S_{N})}$ from where the invertibility of the operator $\mathcal{M}^{12}_{0}$ follows.

Now, we decompose $\mathcal{M}^{12} - \mathcal{M}^{12}_{0}= \mathcal{M}^{12}_{s} + \mathcal{M}^{12}_{c}$ and we prove that $\mathcal{M}^{12}_{0}+ \mathcal{M}^{12}_{s}$ is a compact perturbation of $\mathcal{M}^{12}$. Consequently,  $\mathcal{M}^{12}$ is Fredholm with index zero. In addition, as the operator $\mathcal{M}^{12}$ is one to one, we conclude that it is also continuously invertible. 
\end{proof}

\section{Conclusions}
A new parametrix for the diffusion equation in non homogeneous media (with variable coefficient) has been analysed in this paper. Mapping properties of the corresponding parametrix based surface and volume potentials have been shown in corresponding weigthed Sobolev spaces depending on several regularity and decay conditions on the variable coefficient $a(x)$.  

A BDIES for the original BVP has been obtained. Results of equivalence between the BDIES and the BVP have been shown along with the invertibility of the matrix operator defining the BDIES using Fredholm alternative arguments overcoming the technicalities that unbounded domains present. 

Now, we have obtained an analogous system to the BDIES (M12) of \cite{exterior} with a new family of parametrices which is uniquely solvable. Hence, further investigation about the numerical advantages of using one family of parametrices over another will follow.

Further generalised can be obtained by relaxing the smoothness of the boundary to Lipschitz domains. In this case, one needs the generalised canonical conormal derivative operator defined in \cite{traces, mikhailovlipschitz}. Another possible generalisation could consider relaxing the smoothness of the coefficient, see \cite{nonsmooth}. 

\end{document}